\documentclass{article}  
\usepackage{epsfig}     
\usepackage{amsmath}
\usepackage{makeidx}
\usepackage{amsthm}
\usepackage{amsfonts}
\usepackage{amssymb}
\usepackage{graphics}
\usepackage{enumerate}
\usepackage{multicol}
\usepackage{url}
\newtheorem{theo}{Theorem}[section]
\newtheorem{defin}{Definition}[section]
\newtheorem{question}{Question}

\title{Equivalence of D and D' properties in Banach spaces}

\author{
 Paulo Akira F. Enabe \\
    Escola Politécnica\\
    University of São Paulo\\
  \texttt{paulo.enabe@usp.br}
}

\begin{document}
\maketitle


\begin{abstract}
    \noindent This work explores the equivalence of two sequential properties, $\mathcal{D}$ and $\mathcal{D}'$, for dual Banach spaces under the weak* topology. Property $\mathcal{D}$ ensures that any totally scalarly measurable function is also scalarly measurable, while property $\mathcal{D}'$ states that every weakly* sequentially closed subspace of $X^*$ is weakly* closed. These properties, which are central to the study of the interplay between topology and measurability in Banach spaces, were originally investigated in \cite{plichko2015}, where their equivalence was left as an open question. 

    By examining the topological and measurable structures induced by the Baire $\sigma$-algebra, we prove that properties $\mathcal{D}$ and $\mathcal{D}'$ are indeed equivalent. The proof utilizes the relationship between total sets, weak* closures, and scalar measurability, extending previous results on sequential properties of dual Banach spaces. Additionally, we revisit the failure of property $\mathcal{D}$ in nonseparable Banach spaces with $M$-basic $\ell_1^+$-systems, providing a topological reinterpretation of this phenomenon. These findings contribute to a deeper understanding of the weak* topology and measurable mappings in Banach space theory.
\end{abstract}

%
%

\section{Introduction}
The work of \cite{plichko2015} delves into the sequential properties of dual Banach spaces within the weak* topology. Specifically, it focuses on the relationships between weak* angelicity, property  $\mathcal{E}$ , and properties  $\mathcal{D}$  and  $\mathcal{D}'$. The definitions of those properties are given below. In this text, $X$ denotes a real Banach space and $X^*$ is its dual space. Also, given $T$ a subspace of a topological space $\mathcal{T}$, the sequential closure of $T$ is the collection of limits of all convergent sequences in $T$. The set $T$ is said to be sequentially closed if $T$ coincides with its sequential closure.

\begin{defin}[Property $\mathcal{D}$]
    Let $(\Omega, \Sigma)$ be a measurable space. A Banach space  $X$ is said to satisfy property  $\mathcal{D}$  if every totally scalarly measurable function  $\varphi: \Omega \to X$  is also scalarly measurable.
\end{defin}

\begin{defin}[Property $\mathcal{D}'$]
    A Banach space $X$ has property $\mathcal{D}'$ if every $weakly^*$ sequentially closed subspace of $X^*$ is also $weakly^*$ closed. In other words, a $weakly^*$ sequentially closed subspace $S \subset X^*$ contains all $weak^*$ limits of sequences within $S$.
\end{defin}

Other works have extended the results presented in \cite{plichko2015}. The study \cite{aviles2019} investigates the $weak^*$-sequential properties of Johnson–Lindenstrauss spaces (JL2), focusing on Efremov's property $\mathcal{E}$, which asserts that every element in the $weak^*$-closure of a convex bounded set is the $weak^*$-limit of a sequence within the set. The authors explore the relationships between $\mathcal{E}$, $weak^*$-angelicity of the dual space, and Corson’s property $\mathcal{C}$, demonstrating that these properties form a strict hierarchy. Under the Continuum Hypothesis (CH), they construct Johnson–Lindenstrauss spaces $JL2(F^+)$ and $JL2(F^-)$ associated with maximal almost disjoint (MAD) families of subsets of the natural numbers.

In this context, $F^+$ and $F^-$ represent two distinct MAD families of infinite subsets of $\mathbb{N}$: $F^+$ is carefully constructed to ensure that the associated Johnson–Lindenstrauss space $JL2(F^+)$ satisfies property $\mathcal{E}$, whereas $F^-$ is chosen such that $JL2(F^-)$ fails to satisfy $\mathcal{E}$. This distinction highlights the impact of set-theoretic constructions on the sequential behavior of Banach spaces. By showing that these two spaces exhibit fundamentally different properties under the same axiomatic framework, the authors resolve key questions and emphasize the importance of set-theoretic assumptions in functional analysis.

Building on these ideas, \cite{martinez2023} examines sequential properties of Banach spaces, particularly Corson’s property $\mathcal{C}$. A Banach space $X$ satisfies $\mathcal{C}$ if every point in the $weak^*$-closure of a convex subset of $B_{X^*}$ lies in the $weak^*$-closure of a countable subset of that set. The authors establish equivalences between $\mathcal{C}$, Efremov's properties $\mathcal{E}$ and $\mathcal{E}'$, and $weak^*$-sequential compactness under the Proper Forcing Axiom (PFA), showing that these properties correspond to $B_{X^*}$ having countable tightness. They further prove that property $\mathcal{E}'$ implies $\mathcal{C}$ and demonstrate consistent counterexamples to these implications under the Continuum Hypothesis (CH). By bridging functional analysis and set theory, these works deepen our understanding of $weak^*$-topologies and their sequential structures, resolving longstanding questions while revealing new complexities under different axiomatic frameworks.

Although a question was left unanswered in \cite{plichko2015}:
\begin{question}
    In a Banach space $X$, are properties $\mathcal{D}$ and $\mathcal{D}'$ equivalent?
\end{question}
This text proves that the answer to this question is affirmative. Before proving this result, in Section \ref{sec:properties}, the properties $\mathcal{D}$ and $\mathcal{D}'$ are discussed under their topological aspect. Also, it is shown how this topological aspect relates to the measurable structure. In Section \ref{sec:equivalence}, the equivalence is properly proved. 

In this work, most of the proofs are founded in the theory provided by \cite{rudin1976}, \cite{edgar1977measurability}, \cite{edgar1979measurability}, \cite{brezis2011} and \cite{lax2014}.

\section{Some results regarding properties $\mathcal{D}$ and $\mathcal{D}'$}
\label{sec:properties}
There is an alternative definition for the property $\mathcal{D}$ regarding the Baire $\sigma$-algebra. 

\begin{defin}[Property $\mathcal{D}$]
    A Banach space $X$ is said to have property $D$, if for any total set $\Gamma \subset X^*$, the Baire $\sigma$-algebra generated by $\Gamma$, denoted by $Ba(X,\sigma(X,\Gamma))$, coincides with the Baire $\sigma$-algebra generated by the weak topology, denoted by $Ba(X,w)$:
    \begin{equation}
        Ba(X,\sigma(X,\Gamma)) = Ba(X,w).
    \end{equation}
\end{defin}

This definition relies on the equivalence of the $\sigma$-algebra generated by a total set of functionals $\Gamma$ and the $\sigma$-algebra generated by the weak topology. This equivalence highlights the fact that the topological properties of $X$ in the weak topology fully determine the measurable structure of the space. In particular, this is shown in Theorem \ref{theo:equivalence}. Total sets $\Gamma$ $\subset X^*$ are central because they separate points in $X$, encapsulating the geometric essence of the Banach space through its dual functionals. By requiring that  $Ba(X, \sigma(X, \Gamma)) = Ba(X, w)$  for every total set  $\Gamma$, the definition ensures that this structure is independent of the choice of total set, grounding it in the weak topology of $X$.

Moreover, the use of the Baire $\sigma$-algebra connects this property to the theory of continuous functions, as the Baire $\sigma$-algebra is generated by these functions. This makes the framework inherently topological, focusing on the limits and measurable structures naturally arising from the space’s topology. By linking scalar measurability, which deals with functionals, to topological measurability, the definition provides a unified approach that facilitates deeper analysis of measurable mappings  $\varphi: \Omega \to X$.

\begin{theo}\label{theo:mod_rudin}
    Let $\Gamma \subset X^*$ be a total. If $f\in X^*$ is measurable and $E\subset \mathbb{R}$ is a Borel set, then the preimage of $E$ under $f$, denoted by $f^{-1}(E)$, is measurable in $Ba(X, \sigma (X, \Gamma))$.
\end{theo}

\begin{proof}
    The proof for this theorem can be found in details in \cite{rudin1976}
\end{proof}

\begin{theo}\label{theo:equivalence}
    The following statements are equivalent:
    \begin{enumerate}
        \item For every total set $\Gamma \in X^*$, $Ba(X,\sigma(X,\Gamma))$, coincides with the Baire $\sigma$-algebra in the weak topology $Ba(X,w)$, where $w=\sigma(X,X)$.
        \item For any measurable space $(\Omega, \Sigma)$ and any function $\varphi: \Omega \rightarrow X$, if $\phi$ is totally scalarly measurable, then $\varphi$ is scalarly measurable.
    \end{enumerate}
\end{theo}

\begin{proof}
    Let $\varphi: \Omega \rightarrow X$ be a function from a measurable space $(\Omega, \Sigma)$ to the Banach space $X$, and assume that $\varphi$ is totally scalarly measurable. This means that for a total set $\Gamma \in X^*$, the scalar composition $f \circ \varphi$ is measurable for all $f\in \Gamma$.

    It is known that
    \begin{equation}\label{eq:sigma_equality}
        Ba(X,\sigma(X,\Gamma)) = Ba(X,\sigma(X,w)).
    \end{equation}
    This equivalence implies that the functionals in $\Gamma$ are sufficient to describe the entire weak topology and the Baire $\sigma$-algebra of $X$.

    Since $\varphi$ is totally scalarly measurable, for each $f\in \Gamma$, $f\circ \varphi$ is measurable. By Theorem \ref{theo:mod_rudin}, it is known that $f$ is measurable in $Ba(X,\sigma(X,\Gamma))$. Thus, as $\Gamma$ generates the whole Baire $\sigma$-algebra, $f$ is measurable in the whole $Ba(X,w)$. Finally, it is possible to conclude that $f \circ \varphi$ is measurable and then, $\varphi$ is scalarly measurable with respect to all functionals $X^*$.

    The second part of the proof consists of showing that $\varphi$ is scalarly measurable with respect to a total set $\Gamma \subset X^*$, then equality (\ref{eq:sigma_equality}) holds.

    Assume that the scalar condition holds, i.e., for any measurable space $(\Omega, \Sigma)$, and any function $\varphi: \Omega \rightarrow X$, if $\varphi$ is totally scalarly measurable with respect to the total set $\Gamma \subset X^*$, then it is scalarly measurable with respect to all functions in $X^*$. By this assumption, it is known that $f\circ \varphi$ is measurable for every $f\in\Gamma$. Then, $f\circ \varphi$ is also measurable for every $f\in X^*$.

    By the Baire $\sigma$-algebra $Ba(X,\sigma(X,\Gamma))$ consists of all sets that can be described in terms of the weak topology generated by $\Gamma$. Since $\Gamma$ is total, it separates points in $X$. Thus, if $Ba(X,\sigma(X,\Gamma))$ did not coincide with $Ba(X,w)$, there would exist sets in the Baire $\sigma$-algebra generated by $X^*$ that are not captured by the total set $\Gamma$.

    However, the assumption of the theorem states that if $\varphi$ is totally scalarly measurable with respect to $\Gamma$, it is also scalarly measurable with respect to $X^*$. This means that the measurability condition is fully captured by $\Gamma$, implying that $\Gamma$ generates all the weakly measurable sets. Thus, the Baire $\sigma$-algebra generated by $\Gamma$ must be sufficient to describe all weakly continuous functionals and the weak topology itself. With that, it is possible to conclude \ref{eq:sigma_equality}.
\end{proof}

There are certain Banach spaces that fail Property  $\mathcal{D}$. Theorem \ref{theo:failed_d}, specifically, examines a nonseparable Banach space $X$ equipped with a particular structure called an  $M$-basic  $\ell_1^+$-system, denoted as  $(x_\gamma, f_\gamma)_{\gamma \in \mathcal{N}}$ , where  $\mathcal{N}$  is an uncountable index set. The author in \cite{plichko2015} provides this result by analyzing the interplay between total sets and measurability. In the present work, the result is revisited and interpreted from a topological perspective. The definitions of a $M$-basic system and  $\ell_1^+$-system are given below following the notation presented in \cite{plichko2015}

\begin{defin}[$M$-basic system]
    Let $X$ be a Banach space. A system refers to a collection of pairs $(x_\gamma, f_\gamma)$ such that $x_\gamma \in X$, $f_\gamma \in X^*$, for all $\gamma \in \mathcal{N}$, where $\mathcal{N}$ is a set of indices. A system $(x_\gamma, f_\gamma)$ is said to be $M$-basic if it satisfies the following conditions:
    \begin{enumerate}
        \item For each pair $x_\gamma \in X$ and $f_\beta \in X^*$, it holds that:
        \begin{equation}
            f_\beta (x_\gamma) = \delta_{\beta \gamma} = \begin{cases}
                1, \quad \beta = \gamma, \\
                0, \quad \beta \neq \gamma;
            \end{cases}
        \end{equation}
        \item The set of functionals $\{ f_\gamma \}_{\gamma \in \mathcal{N}}$ is total in the closed linear span $\left[ x_\gamma \right]_{\gamma \in \mathcal{N}}$.
    \end{enumerate}
\end{defin}

\begin{defin}[$\ell_1^+$-system]
    An $M$-basic system is called $\ell_1^+$-system if, for each $\gamma \in \mathcal{N}$:
    \begin{equation}
        \| x_\gamma \| = 1,
    \end{equation}
    and there is a constant $c>0$ such that 
\end{defin}

This theorem establishes that for such Banach spaces, the Baire $\sigma$-algebra generated by a total set $\Gamma \subset X^*$, denoted as  $Ba(X, \sigma(X, \Gamma))$, is strictly smaller than the Baire $\sigma$-algebra generated by the weak topology $Ba(X, w)$. This inequality reflects a fundamental misalignment between the measurable structure induced by the total set $\Gamma$ and the global topological structure of the weak topology. Since property $\mathcal{D}$ requires that $Ba(X, \sigma(X, \Gamma)) = Ba(X, w)$ for all total sets  $\Gamma$ , this result directly implies that $X$ does not satisfy property $\mathcal{D}$.

\begin{theo}\label{theo:failed_d}
    Let $X$ be a nonseparable Banach space with an $M$-basic $\ell^+_1$-system $(x_\gamma, f_\gamma)_{\gamma \in \mathcal{N}}$, where $\mathcal{N}$ is an uncountable set of indexes. Then, it holds that \begin{equation}
        Ba(X,\sigma(X,\Gamma)) < Ba(X,\sigma(X,w)).
    \end{equation}
    Therefore, the Banach space $X$ fails the property $D$.
\end{theo}

\begin{proof}
     Let $(\Omega, \Sigma)$ be a measurable space in which $\Sigma$ contains all singleton subsets of $\Omega$. And let $\mathcal{G}$ be a non-measurable subset of $\Omega$. Define the function $\varphi: \Omega \rightarrow X$ as:
     \begin{equation}
         \varphi (\omega) = 
            \begin{cases}
                x_{\gamma (\omega)}, \quad &if \quad w \in \mathcal{G},\\
                0, \quad &if \quad \omega \not\in \mathcal{G}.
            \end{cases}
     \end{equation}
      Thus, $\varphi$ maps elements of the measurable space $\Omega$ to vectors in the Banach space $X$ by associating vectors from the $M$-basic system to elements of $\mathcal{G}$.

      It is necessary to prove that $\varphi$ is totally scalarly measurable. For each $f_\gamma$ in the system, 
      \begin{equation}
          f_\gamma \circ \varphi (\omega) = 
          \begin{cases}
              f_\gamma(x_{\gamma (\omega)}) = \delta_{\gamma, \gamma(\omega)}, \quad &if \quad \omega \in \mathcal{G},\\
              0, &if \quad \omega \not\in  \mathcal{G}.
          \end{cases}
      \end{equation}
      This means that
      \begin{equation}
          \begin{cases}
              f_\gamma \circ \varphi (\omega) = 1, \quad &if \quad \omega \in \mathcal{G},\\
              f_\gamma \circ \varphi (\omega) = 0, \quad &if \quad \omega \not\in \mathcal{G}.
          \end{cases}
      \end{equation}
      Since $\Sigma$ contains all singletons, and $f_\gamma \circ \varphi$ behaves as the characteristic of the these singletons, the function $f_\gamma \circ \varphi$ is measurable with respect to the $\sigma$-algebra $\Sigma$. The case in which $\omega \not\in \mathcal{G}$, $f_\gamma \circ \varphi \equiv 0$, that is trivially measurable. Therefore, $f_\gamma \circ \varphi$ is a measurable function for each $\omega \in \mathcal{G}$ and $\gamma \in \mathcal{N}$. Finally it is possible to conclude that $\varphi$ is totally scalarly measurable with respect to $\Gamma \subset X^*$.

      Let $A = \overline{conv} \{ x_\gamma \}_{\gamma \in \mathcal{N}}$ be the closed convex hull of the vectors $\{ x_\gamma \}_{\gamma \in \mathcal{N}}$. By the properties of the $l_1^+$-system, the set $A$ is convex, closed and does not contain the zero vector. This is true because the $l_1^+$-system's property ensures that the norm of any non-trivial combination of the vectors remains bounded away from $0$. Then, this implies that no combination of vectors in $A$ can collapse to zero unless all coefficients are zero:
      \begin{equation}
          \left\| \sum \limits_\gamma \alpha_\gamma x_\gamma \right\| \geq c \sum \limits_\gamma c_\gamma.
      \end{equation}

      By the Hanh-Banach Theorem (see Appendix \ref{apx:results}), there exists a functional $g\in X^*$ such that
      \begin{equation}
          g(x) > 0, \quad \forall x\in A.
      \end{equation}
      Thus, the functional $g$ separates the convex set $A$ from the zero vector.

      Now, consider the function $g \circ \varphi$
      \begin{equation}
          g \circ \varphi (\omega)=
          \begin{cases}
              g(x_{\gamma (\omega)} > 0, \quad &if \quad \omega \in \mathcal{G},\\
              g(0) = 0, \quad &if \quad \omega \not\in \mathcal{G}.
          \end{cases}
      \end{equation}
      Since $\mathcal{G}$ is non-measurable (i.e., $\mathcal{G} \not\in \Sigma$), the function $g \circ \varphi$ is also non-measurable with respect to the $\sigma$-algebra $\Sigma$. Hence, $g \circ \varphi$ is not measurable in the Baire $\sigma$-algebra generated by $\Gamma$. Therefore, the non-measurable nature of $g \circ \varphi$ implies that the Baire $\sigma$-algebra generated by the total set $\Gamma \subset X^*$ is strictly smaller than the Baire $\sigma$-algebra generated by the weak topology on $X$.
 \end{proof}

 As can be seen the presence of an $M$-basic system $\ell^+_1$-system in $X$ plays a critical role in the failure of attaining to property $\mathcal{D}$. An $M$-basic system $\ell^+_1$-system provides a structured yet highly flexible set of functionals $f_\gamma$ that separates points in $X$ while interacting intricately with the weak topology. The uncountable index set $\mathcal{N}$ emphasizes the nonseparability of $X$, a property that often introduces complexities in the relationship between topology and measurability. In this setting, the $\sigma$-algebra generated by  $\Gamma$  cannot capture all the sets measurable in the weak topology, leading to the strict inequality between the two Baire $\sigma$-algebras.

 From a topological perspective, this theorem reveals a sharp distinction between separable and nonseparable Banach spaces. In separable Banach spaces, the weak topology is well-aligned with the structures generated by total sets, ensuring that  $Ba(X, \sigma(X, \Gamma)) = Ba(X, w)$. However, in nonseparable spaces with  $M$-basic  $\ell_1^+$-systems, the disjointness between these structures highlights the limitations of total sets in fully representing the topology.

 This result underscores the necessity of considering the interplay between topology and geometry in analyzing measurability properties in nonseparable Banach spaces. Furthermore, it serves as a counterexample to the universality of Property  $\mathcal{D}$, illustrating how specific structural features, such as  $M$-basic  $\ell_1^+$-systems, can lead to its failure. By reinterpreting this result under a topological lens, the present work sheds light on the broader implications of this misalignment between total sets and the weak topology in Banach spaces.

\section{Equivalence of properties $D$ and $D'$}
\label{sec:equivalence}
 \begin{theo}\label{theo:c}
     For a total subspace $F \subset X^*$, the following statements are equivalent:
     \begin{enumerate}
         \item $F$ is $weakly^*$ closed.
         \item For any measurable space $(\Omega, \Sigma)$, if $\varphi: \Omega \rightarrow X$ satisfies $f \circ \varphi$ measurable for all $f \in F$, then $g \circ \varphi$ is measurable for all $g \in X^*$.
     \end{enumerate}
 \end{theo}

 \begin{proof}
     Assuming that $F$ is $weakly^*$ closed. Let $\varphi: \Omega \rightarrow X$ be a function such that $f \circ \varphi $ is measurable for all $f\in F$. Consider $g \in X^*$. Since $F$ is total, it is known that there exists a net $(f_\alpha)_\alpha \subset F$ such that $f_\alpha$ converges to $g$ weakly.

     For any $\omega \in \Omega$, $\varphi(\omega) \in X$, so 
     \begin{equation}
         f_{\alpha}(\varphi (\omega)) \longrightarrow g(\varphi (\omega)).
     \end{equation}
     This means that 
     \begin{equation}
         f_\alpha \circ \varphi \longrightarrow g \circ \varphi
     \end{equation}
     pointwise on $\Omega$. Since $f_\alpha \circ \varphi$ is measurable for all $f_\alpha \in F$, and pointwise limits of measurable functions are measurable, $g \circ \varphi$ is measurable. Thus, $g \circ \varphi$ is measurable for all $g \in X^*$.

     Now, assume that for any measurable space $(\Omega, \Sigma)$ and $\varphi: \Omega \rightarrow X$, if $f \circ \varphi$ is measurable for all $f \in F$, then $g \circ \varphi$ is measurable for all $g \in X^*$. Suppose $F$ is not $weakly^*$ closed. Then, there exists $g\in \overline{F}^*/F$, where $\overline{F}^*$ denotes the $weak^*$ closure of $F$. Let $\Omega = X$ and $\varphi: X \rightarrow X$ be the identity map, i.e., $\varphi(x) = x$, for all $x \in X$. Define the $\sigma$-algebra 
     \begin{equation}\label{eq:sigma_algebra}
         \Sigma = \sigma \left( \{ f \circ \varphi: f \in F \} \right),
     \end{equation}
     which is generated by all measurable sets under $f \circ \varphi$ for $f \in F$. By construction, $f \circ \varphi = f$ is measurable for all $f \in F$, so $f \circ \varphi$ is totally scalarly measurable with respect to $F$. Since $g \not\in F$, there exists a set $E \subset \Omega$ such that $g \circ \varphi = g$ is not measurable with respect to $\Sigma$. This contradicts the assumption that $g \circ \varphi$ must be measurable for all $g \in X^*$. Therefore, $g \not\in F$ is impossible, and $F$ must be $weakly^*$ closed.
 \end{proof}

 \begin{theo}
     For a Banach space $X$, properties $\mathcal{D}$ and $\mathcal{D}'$ are equivalent.
 \end{theo}

\begin{proof}
    Assume $X$ satisfies property $\mathcal{D}'$. Let $\Omega$ be a function that is totally scalarly measurable with respect to a total set $\Gamma \subset X^*$, i.e., $f \circ \varphi$ is measurable for all $f \in \Gamma$. Define 
    \begin{equation}
        F = \{ f \in X^* : f \circ \varphi \quad is \quad measurable\}.
    \end{equation}
    By this definition, $F$ is a total set as $\Gamma \subset F$, and $\Gamma$ is total. Claim that $F$ is $weakly^*$ sequentially closed. Let $(f_n)_{n \in \mathbb{N}}$ $weakly^*$ converge to $f\in X^*$. For any $\omega \in \Omega$, 
    \begin{equation}
        f_{n}(\varphi (\omega)) \longrightarrow f(\varphi (\omega)).
    \end{equation}
    Since each $f_n \circ \varphi$ is measurable, their pointwise limit is also measurable. Thus, $f \in F$, proving that $F$ is $weakly^*$ sequentially closed. By $D'$, $F$ is $weakly^*$ closed. By Theorem \ref{theo:c}, since $F$ is $weakly^*$ closed and $\varphi$ is totally scalarly measurable with respect to $F$, $\varphi$ is scalarly measurable with respect to $X^*$. Therefore, $X$ satisfies property $\mathcal{D}$.

    Now, assume that $X$ satisfies property $\mathcal{D}$. Let $F \subset X^*$ be a total and $weakly^*$ sequentially closed subspace. Let $\Omega=X$ and $\varphi: X \rightarrow X$ be the identity map. Define the $\sigma$-algebra as in (\ref{eq:sigma_algebra}). By construction, for any $f \in F$, $f\circ \varphi = f$ is measurable. Thus, $\varphi$ is totally scalarly measurable with respect to $F$. By $\mathcal{D}$, $g \circ \varphi = g$ is measurable for all $g \in X^*$. By Theorem \ref{theo:c}, $F$ is $weakly^*$ closed. Hence, $X$ satisfies property $D'$.
\end{proof}

\appendix
\section{Results and definitions in Functional Analysis}
\label{apx:results}

In this section, some results and definitions are given with the purpose of standardizing the notation and understanding of the present text. 

\begin{defin}[Total Set]
    Let $X$ be a vector space over a field $\mathbb{F}$, and let $\Gamma \subset X^*$, where $X^*$ denotes the dual space of $X$ (the space of all linear functionals on $X$). The set $\Gamma$ is said to be \textit{total} if the following condition holds:
    \begin{equation}
        \forall x \in X, \quad \big( \varphi(x) = 0 \; \forall \varphi \in \Gamma \big) \implies x = 0,
    \end{equation}
    where $\varphi(x)$ denotes the evaluation of the functional $\varphi$ at the vector $x$. In other words, $\Gamma$ is total if the only vector in $X$ annihilated by every functional in $\Gamma$ is the zero vector.
\end{defin}

\begin{theo}[Hahn-Banach]
    Let $X$ be a real Banach space, and let $A \subset X$ and $B \subset X$ be two disjoint non-empty, convex subsets. Suppose one of those sets is open. Then, there exists a continuous linear functional $\varphi \in X^*$ and $\alpha \in \mathbf{R}$ such that:
    \begin{equation}
        \varphi (a) < \alpha < \varphi (b),
    \end{equation}
    for all $a \in A$ and $b \in B$. In other words, this means that there exists a closed hyperplane
    \begin{equation}
        H = \{ x \in X: \varphi (x) = \alpha \} 
    \end{equation}
    that strictly separates $A$ and $B$.
\end{theo}


\begin{thebibliography}{9}
    \bibitem{plichko2015}
        Plichko, A.,
        \textit{Three sequential properties of dual Banach spaces in the weak* topology},
        Topology and its Applications, Volume 190, 2015, Pages 93--98,
        ISSN 0166-8641, \url{https://doi.org/10.1016/j.topol.2015.04.017}.

    \bibitem{aviles2019}
        Avilés, A., Martínez-Cervantes, G., Rodríguez, J.,
        \textit{Weak*-sequential properties of Johnson–Lindenstrauss spaces},
        Journal of Functional Analysis, Volume 276, Issue 10, 2019, Pages 3051--3066,
        ISSN 0022-1236, \url{https://doi.org/10.1016/j.jfa.2018.09.007}.

    \bibitem{martinez2023}
        Martínez-Cervantes, G., Poveda, A.,
        \textit{On the property (C) of Corson and other sequential properties of Banach spaces},
        Journal of Mathematical Analysis and Applications, Volume 527, Issue 2, 2023, Article 127441,
        ISSN 0022-247X, \url{https://doi.org/10.1016/j.jmaa.2023.127441}.

    \bibitem{rudin1976}
        Rudin, W.,
        \textit{Principles of Mathematical Analysis}, 3rd Edition, McGraw-Hill, New York, 1976, p. 145.

    \bibitem{edgar1977measurability}
        Edgar, G. A.,
        \textit{Measurability in a Banach Space},
        Indiana University Mathematics Journal, Volume 26, Number 4, 1977, Pages 663--677,
        Published by JSTOR.

    \bibitem{edgar1979measurability}
        Edgar, G. A.,
        \textit{Measurability in a Banach Space, II},
        Indiana University Mathematics Journal, Volume 28, Number 4, 1979, Pages 559--579,
        Published by JSTOR.

    \bibitem{brezis2011}
        Brezis, H.,
        \textit{Functional Analysis, Sobolev Spaces and Partial Differential Equations},
        Universitext, Springer, New York, NY, 2011, XIV+600 pages,
        ISBN: 978-0-387-70913-0 (Softcover), 978-0-387-70914-7 (eBook),
        DOI: \url{https://doi.org/10.1007/978-0-387-70914-7}.

    \bibitem{lax2014}
        Lax, P. D.,
        \textit{Functional Analysis},
        Pure and Applied Mathematics: A Wiley Series of Texts, Monographs and Tracts, John Wiley \& Sons, 2014, 608 pages,
        ISBN: 1118626745, 9781118626740.

\end{thebibliography}
\end{document}